\documentclass[a4paper, 10pt]{article}

\usepackage{amsmath,amsthm,latexsym,amscd,amsbsy,amssymb,amsfonts,fleqn,leqno,
euscript}

\usepackage{latexsym}

\usepackage{amssymb}
\newtheorem{thm}{Theorem}[section]

\newtheorem{lemma}[thm]{Lemma}

\begin{document}

\begin{center}
{\Large\bf Having cut-points is not a Whitney reversible property}
\end{center}

\begin{center}
{\bf Eiichi Matsuhashi}
\end{center}

\begin{center}
Department of Mathematics, Faculty of Engineering, Yokohama National University, Yokohama, 240-8501, Japan
\end{center}

\vspace{5mm}

\begin{abstract}
 We show that the property of having cut-points is not a Whitney reversible property. This answers in the negative a question 
posed by Illanes and Nadler.

\end{abstract}


\section{Introduction}
\hspace{1pc}In this note, all spaces are separable metrizable spaces and maps 
are continuous.
We denote the  interval [0, 1] by $I$. A compact metric space 
is called a $compactum$ and $continuum$ means  
a connected compactum. 
If $X$ is a continuum $C(X)$ denotes the space of all subcontinua 
of $X$ with the topology generated by the Hausdorff metric. 

A topological property $P$ is called a $Whitney$ $property$ 
provided that if a continuum $X$ has property $P$, 
so does $\mu^{-1}(t)$ for each Whitney map (see p105 of \cite{nadler2})  
$\mu$ for $C(X)$ 
and for each $t \in [0,\mu(X))$. 
A topological property $P'$ is called a  
$Whitney$ $reversible$ $property$ provided that whenever 
$X$ is a continuum such that $\mu^{-1}(t)$ has property $P'$ for all 
Whitney maps $\mu$ for $C(X)$ and all $t \in (0,\mu(X))$, then 
$X$ has property $P'$. 
These properties have been studied by many authors (see \cite{nadler2}).

A point $p$ in a continuum $X$ is called a $cut$-$point$ $of$ $X$ 
provided that $X  \setminus \{p\}$ is disconnected. 
In this note we 
prove that the property of having cut-points is not a Whitney reversible 
property (it is known that the property of having cut-points is not a 
Whitney property, see Exercise 43.4 of \cite{nadler2}). 
This answers in the negative a question 43.3 of \cite{nadler2} 
posed by Illanes and Nadler.

\section{Main theorem}

A map $f: X \to Y$ between continua is called an $atomic$ $map$ 
if 
$f^{-1}(f(A))=A$ for each $A \in C(X)$ such that $f(A)$ is 
nondegenerate. 

A subcontinuum $T$ of a continuum X is $terminal$, 
if every subcontinuum of X which intersects both $T$ and its complement 
must contain $T$. It is known that a  map $f$ 
of a continuum $X$ onto a continuum $Y$ is 
atomic if and only if every fiber of $f$ is a terminal subcontinuum  of $X$.

The main aim  of this paper is to prove Theorem 2.2. To prove this theorem, 
we need the next result proved by Anderson \cite{lewis1}.

\begin{lemma}{\rm(see THEOREM of  \cite{lewis1})}
For each  continuum Y, there exist a  continuum X 
and a monotone open map $f: X \to Y$ such that $f^{-1}(y)$ is a nondegenerate 
terminal subcontinuum of  $X$ for each $y \in Y$.
\label{lemma:lewis1} 
\end{lemma}

\begin{thm} There exists a continuum $Z$ such that: 

(A) $Z$ does not have a cut-point, and 

(B) $\mu^{-1}(s)$ has a cut-point for each Whitney map 
$\mu: C(Z) \to [0, \mu(Z)]$ and for each $s \in (0, \mu(Z))$. 
\end{thm}

\begin{proof} By Lemma \ref{lemma:lewis1}  there exist a  continuum X 
and a monotone open map $f: X \to I$ such that $f^{-1}(y)$ is a 
nondegenerate terminal subcontinuum of  $X$ for each $y \in I$. Let $Z$ be the quotient space obtained from $X$ by shrinking $f^{-1}(1)$ to 
the point.  Let $p: X \to Z$ be the natural projection and $q=f \circ p^{-1}:Z \to I$. 
Note that $q$ is a monotone open map such that $q^{-1}(1)$ is a one point set 
and $q^{-1}(y)$ is a nondegenerate terminal subcontinuum of $Z$ 
for each $y \in [0,1)$ (hence $q$ is atomic). 

We show that $Z$ has the required properties.  
At first  we prove that $Z$ does not have a cut-point. 
Let $z \in Z$. If $\{z\}= q^{-1}(1)$, then 
$Z \setminus \{z\}= q^{-1}([0,1)).$ Since $q$ is monotone, 
$q^{-1}([0,1))$ is connected. Hence 
in this case $z$ is not a cut-point of $Z$. 
Assume that $z \in q^{-1}(t)$ for some $t \in [0,1)$ and 
$z$ is a cut-point of $Z$. Then there exist non-empty open subsets 
$O,H \subset Z$ such that $Z \setminus \{z\}= O \cup H$ and 
$O \cap H = \emptyset$. Then $O \cup \{z\}$ and $H \cup \{z\}$ are 
nondegenerate 
continua (see Proposition 6.3 of \cite{nadler1}). 
We may assume that $q^{-1}(t) \cap H \neq \emptyset$. 
Since $O \cup \{z\}$ is a nondegenerate continuum, there exists 
$\{z_i\}_{i=1}^{\infty} \subset O$ such that 
$\displaystyle \lim_{i \to \infty} z_i=z$. For 
each $i=1,2,...,$ let $t_i=q(z_i)$. Note that 
$q^{-1}(t_i) \subset O$ for each $i=1,2,...$ Since  $q$ is an open map, 
$\displaystyle \lim_{i \to \infty} q^{-1}(t_i)=q^{-1}(t)$. Then 
$q^{-1}(t) \subset O \cup \{z\}$. This is a contradiction because 
$q^{-1}(t) \cap H \neq \emptyset$. 
 So $Z$ does not have a cut-point.

Next we prove that $\mu^{-1}(s)$ has a cut-point for each Whitney map 
$\mu: C(Z) \to [0, \mu(Z)]$ and for each $s \in (0, \mu(Z))$.  
Take $a,b \in (0,1)$ such that 
$\mu(q^{-1}([a,b]))=s$ (this is possible because $q$ is a monotone 
open map
 and  $q^{-1}(1)$ is a one point set). 
Now we show that 

\vspace{1mm}

(1) $\mu^{-1}(s)= \{C \in \mu^{-1}(s) | C \subset q^{-1}([0,b])\} \cup 
\{C \in \mu^{-1}(s) | C \subset q^{-1}([a,1])\}.$

\vspace{1mm}

If not, there exists  $D \in \mu^{-1}(s)$ such that 
$D \cap q^{-1}([0,a)) \neq \emptyset \neq D \cap q^{-1}((b,1])$. Then 
$q(D)$ contains $[a,b]$ as a proper subcontinuum of $q(D)$. Since $q$ is 
atomic, $D=q^{-1}(q(D))$. So $D$ contains $q^{-1}([a,b])$ as a proper 
subcontinuum of $D$. 
This is a contradiction because $\{D, q^{-1}([a,b])\} \subset \mu^{-1}(s)$. 
Hence  (1) holds. 

It is easy to see that 

\vspace{1mm}

(2) $\{C \in \mu^{-1}(s) | C \subset q^{-1}([0,b])\} \cap 
\{C \in \mu^{-1}(s) | C \subset q^{-1}([a,1])\}= \{q^{-1}([a,b])\}$, and 

\vspace{1mm}

(3) $\{C \in \mu^{-1}(s) | C \subset q^{-1}([0,b])\}$ and 
$\{C \in \mu^{-1}(s) | C \subset q^{-1}([a,1])\}$ are 
nondegenerate subcontinua of $\mu^{-1}(s)$. 

\vspace{1mm}

By (1), (2) and (3) we see that $q^{-1}([a,b])$ is a cut-point of 
$\mu^{-1}(s)$.  \end{proof}

By this result,  we see that the property of having cut-points is not a 
Whitney reversible property.

\vspace{5pc}

Eiichi Matsuhashi
 
Department of Mathematics 

Faculty of Engineering 

Yokohama National University 

Yokohama, 240-8501, Japan

e-mail: mateii@ynu.ac.jp

\end{document}